\documentclass{article}
\usepackage{amsmath}
\usepackage{amssymb}
\usepackage{amsthm}
\usepackage{bbm,mathtools}

\newcommand{\op}[1]{\prescript{o}{}{#1}}
\newcommand{\pp}[1]{\prescript{p}{}{#1}}

\newcommand{\midb}{\;\middle|\;}
\newcommand{\one}{\mathbbm 1}

\def\reals{\mathbb{R}}

\def\gph{\mathop{\rm gph}\nolimits}

\def\comp{\raise 1pt \hbox{$\scriptstyle\circ$}}

\def\esssup{\mathop{\rm ess\ sup}\nolimits}

\def\dom{\mathop{\rm dom}\nolimits}

\def\rge{\mathop{\rm rge}}
\def\Var{\mathop{\rm Var}\nolimits}

\def\upto{{\raise 1pt \hbox{$\scriptstyle \,\nearrow\,$}}}
\def\downto{{\raise 1pt \hbox{$\scriptstyle \,\searrow\,$}}}

\def\cl{\mathop{\rm cl}\nolimits}

\def\FF{(\F_t)_{t\ge 0}}

\def\B{{\cal B}}

\def\D{{\cal D}}

\def\F{{\cal F}}

\def\L{{\cal L}}
\def\M{{\cal M}}
\def\N{{\cal N}}

\def\P{{\cal P}}
\def\R{{\cal R}}

\def\T{{\cal T}}
\def\U{{\cal U}}

\def\Y{{\cal Y}}

\newtheorem{theorem}{Theorem}
\newtheorem{lemma}[theorem]{Lemma}
\newtheorem{corollary}[theorem]{Corollary}

\newtheorem{example}[theorem]{Example}
\newtheorem{remark}{Remark}
\theoremstyle{definition}

\begin{document}
\title{Dual spaces of cadlag processes}

\author{Teemu Pennanen\thanks{Department of Mathematics, King's College London, Strand, London, WC2R 2LS, United Kingdom, teemu.pennanen@kcl.ac.uk} \and Ari-Pekka Perkki\"o\thanks{Mathematics Institute, Ludwig-Maximilian University of Munich, Theresienstr. 39, 80333 Munich, Germany, a.perkkioe@lmu.de}}

\maketitle

\begin{abstract}
This article characterizes topological duals of spaces of cadlag processes. 
We obtain extensions of functional analytic results of Dellacherie and Meyer that underlie many fundamental results in stochastic analysis. In particular, we obtain a characterization of the dual of cadlag processes of class $(D)$ in terms of optional measures of essentially bounded variation. When specialized to regular processes, we find extensions of the main result of Bismut \cite{bis78} on projections of continuous processes. The dual characterizations yield existence results and optimality conditions for many fundamental problems from optimal stopping to singular stochastic control well beyond classical formulations.
\end{abstract}


\section{Introduction}

Many fundamental results in the theory of stochastic processes are based on characterizations of the topological dual of a space of stochastic processes. For example, the Doob decomposition of a supermartingale is obtained by identifying it with a continuous linear functional on a space of bounded cadlag processes and then using the characterization of the dual space; see~\cite[Section~VII.1]{dm82}. The existence of Snell envelope is proved similarly; see~\cite[Appendix~I]{dm82}.  Also, Bismut's characterization of regular processes and existence results on optimal stopping are based on duality; see \cite{bis78,bs77} and the references therein. Moreover, duality theory and optimality conditions for general convex stochastic control problems are often derived in a functional analytic framework of paired spaces of stochastic processes; see e.g.\ \cite{bis73b}. To extend such frameworks to singular stochastic control, one needs processes of bounded variation (BV) in separating duality with a space of cadlag processes; see e.g. \cite{pp17}. 

This paper studies Fr\'echet spaces (in particular, Banach spaces) of stochastic processes whose dual can be identified with a space of optional Borel measures (and thus with BV-processes). When dealing with {\em raw} (not necessarily adapted) stochastic processes, the duality is fairly easy to establish. The dual of the space $\Y(D)$ of raw cadlag processes whose pathwise supremum norms belong to a Fr\'echet space $\Y$ of random variables turns out to be a space of pairs of random measures whose pathwise total variation belongs to the dual of $\Y$. When specialized to continuous processes, each dual element can be represented by a single random measure.

The case of adapted cadlag processes is more involved and requires additional techniques. This paper characterizes topological duals of spaces of adapted cadlag processes via functional analysis of the optional projection on the space $\Y(D)$ of raw stochastic processes. Our main results are based on the classical closed range and closed graph theorems which are valid in general Fr\'echet spaces; see e.g.~\cite{kn76}. We find that, as soon as the optional projection is continuous to a Fr\'echet space $\D$ of adapted cadlag processes, its surjectivity is equivalent the topological dual of $\D$ being identifiable with a space of optional random measures. The equivalence is then applied in two directions.

When $\D$ is a subspace of $\Y(D)$, the surjectivity is clear, so we recover \cite[Theorem~VII.65]{dm82}, where $\Y=L^p$ with $p>1$, and its extension \cite[Theorem~3.1]{ara14} where $\Y$ is the Morse heart of an appropriate Orlicz space. We obtain a further extension where $\Y$ is a symmetric (rearrangement invariant) Banach function space whose upper Boyd index is strictly less than one. Banach function spaces are natural extensions of $L^p$ and Orlicz spaces and they have been studied extensively since they were introduced in Luxemburg; see \cite{lux55,lz71,kps82,bs88}. Boyd indices were used e.g.\ in \cite{nov91,kik99} to extend martingale inequalities to Banach function spaces.

In many interesting cases (e.g., when $\Y=L^1$), the optional projection takes a process out of $\Y(D)$. When the polar seminorms of the Fr\'echet space $\Y$ can be expressed as Choquet integrals, we can specify $\D$ so that its dual can be identified directly with a space of optional random measures. The Choquet property holds, in particular, when $\Y=L^1$, so we recover \cite[Theorem~4]{bis78} on the surjectivity of the optional projection of $L^1(D)$ to the space $\D^1$ of cadlag processes of class $(D)$. Our general result then identifies the topological dual of $\D^1$ in terms of optional measures of essentially bounded variation. This characterization seems new. Our arguments are considerably simpler than those of \cite{bis78} and they also cover Marcinkiewicz spaces (see e.g.\ \cite{kps82,bs88}) for which both the surjectivity and the duality results are new.

The main result also gives surjectivity and duality results on general spaces of {\em regular} processes, i.e.\ adapted cadlag processes $y$ of class $(D)$ whose predictable projections coincide with their left limits. In particular, we recover the main result of \cite{bis78} on the surjectivity of the optional projection from the space $L^1(C)$ of integrable continuous processes to the space $\R^1$ of regular processes. As a byproduct, we recover \cite[Theorem~8]{pp17} characterizing the topological dual of $\R^1$ as the space of optional measures of essentially bounded variation. Both results are extended by replacing $L^1$ by a Fr\'echet space $\Y$ as in the main theorems on cadlag processes.

The provided extensions are of interest e.g.\ in singular stochastic control where one optimizes over spaces of optional processes of bounded variation. Our results allow for formulations where the variation need not be essentially bounded. One can then develop dual problems and optimality conditions for convex singular stochastic control much like in \cite{bis73b} in the case of absolutely continuous trajectories. This will be developed in a separate article.

The rest of this paper is organized as follows. Section~\ref{sec:frechet} gives a quick review on topological duals of Fr\'echet spaces of random variables. Section~\ref{sec:raw} characterizes the topological dual of raw (nonadapted) cadlag processes whose pathwise supremum belongs to a Fr\'echet space of random variables. Section~\ref{sec:acp} gives the main result of the paper stating that the dual of a space $\D$ of adapted cadlag processes has the desired structure if and only if the optional projection is a continuous surjection on $\D$. Sections~\ref{sec:doob}, \ref{sec:choquet} and \ref{sec:regular} then derive some fundamental known and new duality results in the theory of stochastic processes. Section~\ref{sec:dd} gives a further application to decomposition of semimartingales.

\section{Fr\'echet lattices of random variables}\label{sec:frechet}

Let $(\Omega,\F,P)$ be a probability space. The duality theory for stochastic processes developed in this paper assumes a Fr\'echet space $\Y$ of random variables whose topological dual can be identified with another space $\U$ of random variables in the sense that every element of the dual can be expressed uniquely as
\[
\langle \xi,\eta\rangle := E(\xi\eta)\quad \xi\in\Y
\]
for some $\eta\in\U$. This section gives a large class of such spaces $\Y$ along with some well-known examples.

Let $\P$ be a countable collection of sublinear symmetric functions on $L^1:=L^1(\Omega,\F,P)$ and define
\[
\tilde \Y:=\bigcap_{p\in\P}\dom p.
\]
We endow $\tilde \Y$ with the locally convex topology generated by $\P$ and assume that each $p\in\P$ satisfies the following:
\begin{enumerate}
\item $p$ is lower semicontinuous on $L^1$,
\item there exists a constant $c$ such that $\frac{1}{c}\|\xi\|_{L^1}\le p(\xi)\le c\|\xi\|_{L^\infty}$ for all $\xi\in L^1$,
\item $p(\xi_1)\le p(\xi_2)$ whenever $|\xi_1|\le|\xi_2|$,
\item $p(\xi^\nu)\downto 0$ whenever $(\xi^\nu)_{\nu=1}^\infty\subset L^\infty$ with $\xi^\nu\downto 0$ almost surely.
\end{enumerate}
We define $\Y$ as the closure of $L^\infty$ in $\tilde \Y$. The above setting covers, in particular, $L^p$ spaces with $p\in[1,\infty)$. Indeed, when $\P=\{\|\cdot\|_{L^p}\}$, we have $\Y=\tilde \Y=L^p$. More interesting examples will be given at the end of this section.

For each $p\in\P$, we define a sublinear symmetric function $p^\circ$ on $L^1$ by
\[
p^\circ(\eta):=\sup_{\xi\in L^\infty}\{E(\xi\eta)\,|\,p(\xi)\le 1\}.
\]
The following is from \cite{pp20}.

\begin{theorem}\label{thm:fl}
The space $\Y$ is Fr\'echet and its dual may be identified with the space 
\[
\U:=\bigcup_{p\in\P}\dom p^\circ
\]
under the bilinear form
\[
\langle \xi,\eta\rangle := E(\xi\eta).
\]
For every $\xi\in L^1$ and $\eta\in L^1$,
\[
E(\xi\eta)\le p(\xi)p^\circ(\eta),
\]
so $p^\circ$ is the polar of $p$ on $\U$. Each $p^\circ$ satisfies 1,2, and 3.
\end{theorem}

\begin{remark}\label{rem:trunc}
For any $\xi\in\Y$, the pointwise projection $\xi^\nu$ of $\xi$ to $[-\nu,\nu]$ converges to $\xi$ in $\Y$ as $\nu\to\infty$. Indeed, given an $\epsilon>0$, there exists $\bar \xi \in L^\infty$ such that $p(\xi-\bar \xi)<\epsilon$, so, by 3 and 4,  
\[
p(\xi^\nu-\xi)\le p(\one_{|\xi|\ge \nu}\xi ) \le p (\one_{|\xi|\ge \nu}(\xi-\bar\xi) ) + p (\one_{|\xi|\ge \nu}\bar\xi ) \le   p (\xi-\bar\xi) + p (\one_{|\xi|\ge \nu}\bar\xi ) < \epsilon
\]
for $\nu$ large enough. In particular,
\[
\Y=\{\xi\in\tilde \Y \mid \lim_{\nu\to\infty} p(\one_{|\xi|\ge \nu}\xi)= 0\}.
\]
\end{remark}

\begin{example}[Banach function spaces]\label{ex:fbfs}
When $\P$ is a singleton, we are in the setting of Banach function spaces and Theorem~\ref{thm:fl} characterizes the topological dual of $\Y$ as its Koethe dual aka associate space; see e.g.\ \cite{lux55,bs88,kps82}.  
\end{example}

Short proofs of the following (and more) can be found in \cite{pp20}; see also \cite{lux55,kps82,bs88}.

\begin{example}[Orlicz spaces]\label{ex:orl}
Let $\Phi$ be a nonzero nondecreasing finite convex function on $\reals_+$ with $\Phi(0)=0$ and assume that $\P$ contains only the Luxemburg norm
\[
p(\xi) :=\inf\{\beta>0 \mid E\Phi(|\xi|/\beta)\le 1\}.
\]
Then $\tilde\Y$ is the Orlicz space associated with $\Phi$, $p$ satisfies 1-4 and 
\[
\Y=\{\xi\in L^1 \mid E\Phi(|\xi|/\beta)< \infty\quad \forall\beta>0\},
\]
the associated Morse heart. The polar of $p$ can be expressed as
\[
p^\circ(\eta)=\sup_{\xi\in L^\infty}\{E(\eta\xi)\mid E\Phi(\xi)\le 1\} = \inf_{\beta>0}\{\beta E\Phi^*(\eta/\beta)+\beta\}
\]
and, moreover,
\[
\|\eta\|_{\Phi^*}\le p^\circ(\eta)\le 2\|\eta\|_{\Phi^*},
\]
where $\|\cdot\|_{\Phi^*}$ is the Luxemburg norm associated with the conjugate of $\Phi$. Thus, the dual of $\Y$ coincides with the Orlicz space
\[ 
\U=\{\eta\in L^1\mid \exists\beta>0:\ E\Phi^*(\eta/\beta)<\infty\}.
\]
\end{example}

Given $\xi\in L^1$, let
\[
n_\xi(\tau):=E1_{\{|\xi|>\tau\}}
\]
and
\[
q_\xi(t):=\inf\{\tau\in\reals\mid n_\xi(\tau)\le t\}.
\]
Note that $\tau\mapsto 1-n_\xi(\tau)$ is the cumulative distribution function of $|\xi|$ and that $q_\xi$ is an inverse of $n_\xi$. Both $n_\xi$ and $q_\xi$ are nonincreasing. Recall that a probability space is {\em resonant} if it is atomless or completely atomic with all atoms having equal measure. 

\begin{example}[Lorentz and Marcinkiewicz spaces]\label{ex:ml}
Let $\phi$ be a nonnegative concave increasing function on $[0,1]$ and assume that $\P$ contains only the {\em Marcinkiewicz norm}
\[
p(\xi) := \sup_{t\in(0,1]}\left\{\frac{1}{\phi(t)}\int_0^tq_\xi(s)ds\right\}.
\]
Assume that $(\Omega,\F,P)$ is resonant and that $\lim_{t\searrow 0}t/\phi(t)=0$. Then $\tilde\Y$ is the {\em Marcinkiewicz space} associated with $\phi$, $p$ satisfies 1-4 and 
\[
\Y =\{\xi\in L^1\mid \lim_{t\searrow 0} \frac{1}{\phi(t)}\int_0^t q_\xi(s)ds=0\}.
\]
The polar of $p$ can be expressed as
\[
p^\circ(\eta) = \int_0^1q_\eta(t)d\phi(t).
\]
Thus, the dual of $\Y$ coincides with the {\em Lorentz space}
\begin{align*}
\U=\{\eta\in L^1\mid \int_0^1q_\eta(t)d\phi(t)<\infty\}.
\end{align*}
\end{example}

\section{Raw cadlag processes}\label{sec:raw}

This section characterizes the topological dual of a Fr\'echet space of raw (not necessarily adapted) cadlag processes. This will provide the basis for the duality theory of adapted cadlag processes in the subsequent sections. We will assume from now on that $(\Omega,\F,P)$ is complete.

The Banach space of cadlag functions on $[0,T]$ equipped with the supremum norm will be denoted by $D$. We allow $T=+\infty$ in which case $[0,T]$ is understood as the one point compactification of the positive reals. The spaces of Borel measures and purely discontinuous Borel measures on $[0,T]$ will be denoted by $M$ and $\tilde M$, respectively. The dual of $D$ can be identified with $M\times\tilde M$ through the bilinear form
\[
\langle y,(u,\tilde u)\rangle := \int ydu + \int y_-d\tilde u
\]
and the dual norm is given by 
\[
\sup_{y\in D}\{\int ydu + \int y_-d\tilde u\,|\,\|y\|\le 1\}=\|u\|+\|\tilde u\|,
\]
where $\|u\|$ denotes the total variation norm on $M$. This can be deduced from \cite[Theorem~1]{pes95} or seen as the deterministic special case of \cite[Theorem~VII.65]{dm82} combined with \cite[Remark~VII.4(a)]{dm82}.

We assume that $\Y$ and $\U$ are as in Section~\ref{sec:frechet} and define
\[
\Y(D):=\{y\in L^1(D)\mid \|y\|\in\Y\},
\]
where $L^1(D)$ is the space of cadlag processes $y$ with $E\|y\|<\infty$. Throughout, we identify processes that coincide almost surely everywhere on $[0,T]$. We equip $\Y(D)$ with the topology induced by the seminorms 
\[
y\mapsto p(\|y\|),\quad p\in\P.
\]

\begin{theorem}\label{thm:raw}
The space $\Y(D)$ is Fr\'echet and its dual can be identified with
\[
\U(M\times\tilde M) := \{(u,\tilde u)\in L^1(M\times\tilde M)\mid \|u\|+\|\tilde u\|\in\U \}
\]
through the bilinear form
\[
\langle y,(u,\tilde u)\rangle := E\left[\int ydu+\int y_-d\tilde u\right].
\]
Moreover, $L^\infty(D)$ is dense in $\Y(D)$, for every $y\in L^1(D)$ and $(u,\tilde u)\in L^1(M\times \tilde M)$,
\begin{equation}
E\left[\int ydu+\int y_-d\tilde u\right] \le p(\|y\|)p^\circ(\|u\|+\|\tilde u\|)\label{eq:cs}
\end{equation}
and
\begin{equation}\label{eq:pptv}
p^\circ(\|u\|+\|\tilde u\|) = \sup_{y\in L^\infty(D)}\left\{E\left[\int ydu +\int y_-d\tilde u \right] \midb p(\|y\|)\le 1\right\}.
\end{equation}
In particular, $(u,\tilde u)\mapsto p^\circ(\|u\|+\|\tilde u\|)$ is the polar of $y\mapsto p(\|y\|)$.
\end{theorem}
\begin{proof}
We start by showing that $\tilde\Y(D):=\{y\in L^1(D)\mid \|y\|\in\tilde\Y\}$ is complete under the topology induced by the seminorms $y\mapsto p(\|y\|)$. If $(y^\nu)$ is a Cauchy sequence in $\tilde\Y(D)$, it is, by Property 2, Cauchy also in $L^1(D)$ which is complete (see e.g.\ \cite[Theorem~VI.22]{dm82}), so $(y^\nu)$ $L^1(D)$-converges to an $y\in L^1(D)$. Being Cauchy in $\tilde\Y(D)$ means that for every $\epsilon>0$ and $p\in\P$, there is an $N$ such that
\[
p(\|y^\nu-y^\mu\|)\le\epsilon\quad\forall\nu,\mu\ge N.
\]
By the triangle inequality and property 3 of $p$,
\[
p(\|y^\nu-y\|-\|y-y^\mu\|)\le\epsilon\quad\forall\nu,\mu\ge N.
\]
Letting $\mu\to\infty$ and using property 1 now gives
\[
p(\|y^\nu-y\|)\le\epsilon\quad\forall\nu\ge N.
\]
Since $p\in\P$ and $\epsilon>0$ were arbitrary, we thus have $y\in\tilde\Y(D)$ and that $(y^\nu)$ converges in $\tilde\Y(D)$ to $y$. Thus $\tilde\Y(D)$ is complete. It is clear that $\Y(D)$ contains the closure of $L^\infty(D)$. On the other hand, given $y\in\Y(D)$, its pointwise projection $y^\nu$ to the interval $[-\nu,\nu]$ belongs to $L^\infty(D)$ and, by Remark~\ref{rem:trunc}, converges to $y$. Thus $\Y(D)$ is a closed subspace of a Fr\'echet space and thus, Fr\'echet as well.

We have
\[
\left[\int ydu+\int y_-d\tilde u\right]\le \|y\|(\|u\|+\|\tilde u\|)
\]
almost surely, so \eqref{eq:cs} follows from Theorem~\ref{thm:fl}. Every element of $\U(M\times\tilde M)$ thus defines a continuous linear functional on $\Y(D)$. Conversely, a continuous linear functional $J$ on $\Y(D)$ satisfies property (5.1) in \cite[Section~VII.5]{dm82} so, as in the proof of \cite[Theorem~VII.65]{dm82}, there exists $(u,\tilde u)\in L^1(M\times\tilde M)$ such that
\[
J(y) = E\left[\int ydu+\int y_-d\tilde u\right]
\]
on $L^\infty(D)$. Given $\delta\in(0,1)$, a measurable selection argument gives the existence of a $y\in L^1(D)$ such that
\[
\|y\|\le 1\quad\text{and}\quad \int ydu +\int y_-d\tilde u\ge\delta(\|u\|+\|\tilde u\|)
\]
almost surely\footnote{Indeed, $(D,\B(D))=(S,\B(S))$, where $S$ is the space of cadlag functions equipped with the Skorokhod topology. The set 
\[
G:=\{(y,\omega)\in S\times\Omega \mid \|y\|\le 1, \int ydu(\omega) +\int y_-d\tilde u(\omega)\ge\delta(\|u(\omega)\|+\|\tilde u(\omega)\|\}
\]
is $\B(S)\otimes\F$-measurable (see the proof of \cite[Lemma~3]{pp18}) and each $\omega$-section of $G$ is nonempty. Thus \cite[Theorem III.18]{cv77} gives the existence of a measurable selection.}.
Thus, for any $p\in\P$ and $\xi\in L^\infty_+$ such that $p(\xi)\le 1$, 
\begin{align*}
  E[\xi(\|u\|+\|\tilde u\|)] &\le E[\int (\xi y)du +\int (\xi y_-)d\tilde u]/\delta\\
&\le \sup_{y\in L^\infty(D)}\left\{E\left[\int ydu +\int y_-d\tilde u \right] \midb p(\|y\|)\le 1\right\}/\delta.
\end{align*}
The definition of $p^\circ$ now gives $p^\circ(\|u\|+\|\tilde u\|)\le p_{\Y(D)}^\circ(J)/\delta$. Since $\delta\in(0,1)$ was arbitrary, we see that the left hand side of \eqref{eq:pptv} is less than the right side. The reverse follows from \eqref{eq:cs}.
\end{proof}

\section{Adapted cadlag processes}\label{sec:acp}

This section starts with some useful observations concerning optional and predictable projections. We then give our first main result which gives a necessary and sufficient condition for the topological dual of a space of adapted cadlag processes to be representable by random measures.

Let $\FF$ be a filtration satisfying the usual conditions. The set of stopping times will be denoted by $\T$. A measurable process $y$ is said to be of {\em class $(D)$} if the set $\{y_\tau\mid \tau\in\T\}$ is uniformly integrable.  Given such a $y$, we will denote its {\em optional and predictable projections} by $\op y$ and $\pp y$, respectively. That is, $\op y$ is the unique optional process satisfying
\[
E[y_\tau\one_{\{\tau<\infty\}}\mid \F_\tau] = \op y_\tau\one_{\{\tau<\infty\}}\quad P\text{-a.s.}
\]
for every $\tau\in\T$ while $\pp y$ is the unique predictable process satisfying
\[
E\left[y_\tau\one_{\{\tau<\infty\}}\mid \F_{\tau-}\right]=\pp y_\tau\one_{\{\tau<\infty\}}\quad P\text{-a.s.}
\]
for every predictable time $\tau$. Here $\F_\tau:=\sigma(A \mid A\cap\{\tau\le t\}\in\F_t\ \forall\ t\}$ and $\F_{\tau-}:=\F_0\vee\sigma\{A\cap\{t<\tau\}\mid A\in\F_t,\, t\in\reals_+\}$. Throughout the paper, we identify processes that are equal almost surely everywhere, that is, $y^1=y^2$ if, almost surely, $y^1_t=y^2_t$ for all $t$.


By \cite[Remark~VI.50.(f)]{dm82}, the optional projection of a cadlag process of class $(D)$ is a cadlag process of class $(D)$ while the predictable projection of a caglad process of class $(D)$ is a caglad process of class $(D)$.

\begin{lemma}\label{lem:oppp}
For any cadlag process $y$ of class $(D)$, we have $(\op y)_-=\pp(y_-)$.
\end{lemma}

\begin{proof}
Given bounded predictable time $\tau$, it is enough to verify, by the predictable section theorem \cite[Corollary~4.11]{hwy92}, that  $((\op y)_-)_\tau=\pp(y_-)_\tau$. By \cite[Theorem~4.16]{hwy92}, there is a sequence $(\tau^\nu)$ of stopping times with $\tau^\nu<\tau^{\nu+1}$ and $\tau^\nu\nearrow\tau$ almost surely. Let $A_\nu\in\F_{\tau^\nu}$, and $\tau^\nu_j:=\tau ^{\nu+j}$ on $A_\nu$ and $\tau^\nu_j:=+\infty$ otherwise. We have $A_\nu\in\F_{\tau^{\nu+j}}$ for each $j$ \cite[Theorem 3.4]{hwy92}, so $\tau^\nu_j$ are stopping times \cite[Theorem 3.9]{hwy92}. Since $y$ and $\op y$ are of class $(D)$,
\begin{align*}
E [(\op y)_{\tau-}\one_{A_\nu}]=\lim_j E [(\op y)_{\tau^\nu_{j}}\one_{A_{\nu}}]=\lim_j E [y_{\tau^\nu_{j}}\one_{A_\nu}] = E [y_{\tau-}\one_{A_\nu}].
\end{align*}
By \cite[Theorem 3.6]{hwy92}, $\F_{\tau-}=\bigvee_\nu \F_{\tau^\nu}$, which proves the claim, since $A\in\F_{\tau^\nu}$ was arbitrary.
\end{proof}

Given $(u,\tilde u)\in L^1(M\times\tilde M)$, there exist $u^o\in L^1(M)$ and $\tilde u^p\in L^1(\tilde M)$ such that for every bounded measurable process $y$,
\begin{align*}
E\int\op ydu &= E\int ydu^o,\\
E\int\pp ydu &= E\int ydu^p;
\end{align*}
see, e.g., \cite[Remark~VI.74(b)]{dm82}. The random measure $u^o$ is called the {\em optional projection} of $u$ while $\tilde u^p$ is called the {\em predictable projection} of $\tilde u$. One says that $u$ is {\em optional} if $u=u^o$ and that $\tilde u$ is {\em predictable} if $\tilde u=\tilde u^p$.

From now on, $\Y$ and $\U$ are as in Section~\ref{sec:frechet}. The optional projection is a linear mapping from $\Y(D)$ to the space of adapted cadlag processes of class $(D)$. We denote this linear mapping by $\pi$ and its kernel by 
\[
\ker\pi :=\{y\in\Y(D)\mid \op y=0\}.
\]
The space
\[
\hat\M:=\{(u,\tilde u)\in\U(M\times\tilde M)\mid u=u^o,\ \tilde u=\tilde u^p\}
\]
will play a central role in the remainder of this paper. Indeed, we will find it as the topological dual of spaces of adapted cadlag processes. 

\begin{lemma}\label{lem:ope}
For every $(u,\tilde u)\in\hat\M$ and $y\in \Y(D)$,
\begin{align*}
E\int y du  &= E\int \op y du,\\
E\int y_- d\tilde u &= E\int(\op y)_- d\tilde u.
\end{align*}
\end{lemma}
\begin{proof}
Let $u$ be nonnegative. By \cite[Theorem 5.16]{hwy92}, $E\int y du= E\int \op y du$ if $y$ is nonnegative. The random variables $\int y^+ du, \int \op y^+ du, \int y^- du$ and $\int \op y^- du$ are integrable almost surely, so
\[
E\int y du = E\int (y^+-y^-) du =  E\int \op y^+du - E\int \op y^- du = E\int \op ydu.
\]
For general $u$, the claim follows by taking differences. The second equality is proved similarly after noting that $(\op y)_-=\pp(y_-)$, by Lemma~\ref{lem:oppp}.
\end{proof}

The following characterizes $\hat\M$ in terms of the pairing of $\Y(D)$ and $\U(M\times\tilde M)$ obtained in Theorem~\ref{thm:raw} above. Given $u\in M$, we denote by $u^c$ and $u^d$ the continuous and purely discontinuous parts of $u$, respectively.


\begin{lemma}\label{lem:opm}
The space $\hat\M$ is the orthogonal complement of $\ker\pi$ and thus, weakly closed in $\U(M\times\tilde M)$.
\end{lemma}

\begin{proof}
By Lemma~\ref{lem:ope}, $\hat\M$ is contained in the orthogonal complement of $\ker\pi$. It thus suffices to show that $\hat\M$ contains the orthogonal complement of $\ker\pi\cap L^\infty(D)$. We have that $u\in\U(\hat\M)$ belongs to this complement if and only if
\begin{equation}\label{eq:vc}
E\left[\int (y-\op y)du+\int (y-\op y)_-d\tilde u\right] = 0\quad \forall y\in L^\infty(D).
\end{equation}
The equation can be written as
\begin{align*}
0&=E\left[\int (y-\op y)du+\int (y_--\pp (y_-))d\tilde u\right]\\
&=E\left[\int yd(u-u^o)+\int y_- d(\tilde u-\tilde u^p)\right]\\
&=E\left[\int yd(u-u^o-(\tilde u^p)^c)+\int y_- d(\tilde u-(\tilde u^p)^d)\right].
\end{align*}
Since $L^\infty(D)$ is dense in $\Y(D)$, the variational condition implies, by Theorem~\ref{thm:raw}, that $\tilde u-(\tilde u^p)^d=0$ and $u-u^o-(\tilde u^p)^c=0$. The first equation implies that $\tilde u$ is predictable and that $(\tilde u^p)^d=0$. The second equation then implies that $u$ is optional.
\end{proof}

The following is the first main result of this paper. It will yield, later on, characterizations of topological duals of various more specific spaces of adapted cadlag processes.

\begin{theorem}\label{thm:abs}
Let  $\D$ be a Fr\'echet space of adapted cadlag processes. The following are equivalent:
\begin{enumerate}
\item the optional projection is a continuous surjection from $\Y(D)$ to $\D$,
\item  $\op \Y(D)\subset \D$ and the dual of $\D$ can be identified with $\hat\M$ under the bilinear form
\[
\langle y,(u,\tilde u)\rangle =E\left[\int ydu +\int y_-d\tilde u\right].
\]
\end{enumerate}
In this case, the adjoint of the projection is the embedding of $\hat\M$ to $\U(M\times\tilde M)$, $\hat\M=(\ker\pi)^\perp$ and the topology of $\D$ is generated by the seminorms
\[
p_\D(y) := \inf_{z\in\Y(D)} \{p(\|z\|) \mid \op z = y\} \quad p\in \P
\]
whose polars are given by
\[
p_\D^\circ ((u,\tilde u))= p^\circ(\|u\|+\|\tilde u\|).
\]
\end{theorem}

\begin{proof}
Assume 1. Continuity of $\pi$ implies that $\ker\pi$ is closed, so, by \cite[Theorem~12.14.9]{die76}, $\Y(D)/\ker\pi$ is a Fr\'echet space. Thus, by \cite[Theorem~11.2]{kn76} (a corollary of the closed graph theorem), $\D$ is isomorphic to the quotient space $\Y(D)/\ker\pi$. By \cite[Theorem 17.14(ii)]{kn76}, the dual of $\Y(D)/\ker\pi$ can be identified with the orthogonal complement $(\ker\pi)^\perp$ of $\ker\pi$ on the dual of $\Y(D)$ which, by Theorem~\ref{thm:raw}, is $\U(M\times\tilde M)$. By Lemma~\ref{lem:opm}, $(\ker\pi)^\perp=\hat\M$, so 2 holds.

Assume now 2. By the closed graph theorem, $\pi:\Y(D)\to \D$ is continuous if it has a closed graph. Given $u\in\U(\hat M)$ and $\bar u\in \hat\M$, Lemma~\ref{lem:ope} implies
\begin{align*}
(\gph\pi)^\perp &=\{(u,\bar u)\in\U(\hat M)\times\hat\M\mid \langle y,u\rangle+\langle \pi y, \bar u\rangle = 0\ \forall y\in\Y(D)\}\\
&=\{(u,\bar u)\in\U(\hat M)\times\hat\M\mid \langle y,u\rangle+\langle y, \bar u\rangle = 0\ \forall y\in\Y(D)\}\\
& = \{(u,\bar u)\in\U(\hat M)\times\hat\M\mid u=-\bar u\},
\end{align*}
so, since $\hat\M\subset \U(\hat M)$, the bipolar theorem and Lemma~\ref{lem:ope} give
\begin{align*}
\cl \gph \pi &= \{(y,\bar y) \in \Y(D)\times\D \mid \langle y,-u\rangle + \langle\bar y,u\rangle = 0 \ \forall u\in\hat\M \}\\
&= \{(y,\bar y) \in \Y(D)\times\D \mid \langle\op y,-u\rangle + \langle\bar y,u\rangle = 0 \ \forall u\in\hat\M \}\\
&=\{(y,\bar y )\in \Y(D)\times\D \mid \op y=\bar y\}\\
&=\gph \pi
\end{align*}
so the graph is closed and $\pi$ is continuous. The above also shows that the adjoint $\pi^*$ is indeed the embedding. Since $(\rge\pi)^\perp=\ker\pi^*=\{0\}$, the bipolar theorem gives $\cl\rge\pi=\D$, so it suffices to show that $\rge\pi$ is closed. By the closed range theorem \cite[Theorem 21.9]{kn76}, this is equivalent to $\rge\pi^*$ being closed in $\U(M\times\tilde M)$. Since $\pi^*$ is the embedding, its range is $\hat\M$ which is closed, by Lemma~\ref{lem:opm}.

The isomorphism of $\D$ and $\Y(D)/\ker\pi$ also implies that the topology of $\D$ is induced by the quotient space seminorms $p_\D$. Since the adjoint of the optional projection is the embedding of $\hat\M$, the polar of $p_\D$ can be expressed for every $(u,\tilde u)\in\hat\M$ as
\begin{align*}
p_\D^\circ((u,\tilde u)) &= \sup_{y\in\D}\{\langle y,(u,\tilde u)\rangle \,|\, \inf_{z\in\Y(D)}\{p(\|z\|)\,|\, \op z=y\}\le 1\}\\
&= \sup_{z\in\Y(D)}\{\langle\op z,(u,\tilde u)\rangle \,|\, p(\|z\|)\le 1\}\\
&= \sup_{z\in\Y(D)}\{\langle z,(u,\tilde u)\rangle \,|\, p(\|z\|)\le 1\}\\
&= p^\circ(\|u\|+\|\tilde u\|),
\end{align*}
where the last equality follows from Theorem~\ref{thm:raw}.
\end{proof}

\begin{remark}
By the bipolar theorem,
\[
p_\D(y) = \sup\{\langle y,(u,\tilde u)\rangle \,|\, p_\D^\circ((u,\tilde u))\le 1\}.
\]
Restricting the supremum to $\tilde u=0$ and $u$ that only has mass at $T$, gives $p_\D(y)\ge p(y_T)$. On the other hand, if $y$ is a martingale, it is the optional projection of $y_T\one$, so $p_\D(y)\le p(y_T)$. Thus, $p_\D(y)=p(y_T)$ if $y$ is a martingale. Moreover, by the closed range theorem, the set of martingales is a closed subspace of $\D$ and its dual can be identified with $\U$ via the bilinear form $\langle y,\eta \rangle =E[y_T\eta]$.
\end{remark}

Note also that $y\in\Y(D)$ does not imply $\op y\in\Y(D)$, in general. In other words, the optional projection need not be a projection in the sense of functional analysis. Indeed, if $y$ is a martingale, it is the optional projection of the constant process $\one y_T\in L^1(D)$ but it may happen that $\|y\|\notin L^1$. Similarly, $(u,\tilde u)\in\U(M\times\tilde M)$ does not imply $(u^o,\tilde u^p)\in\U(M\times\tilde M)$, in general. 

\begin{example}
Let $y\in L^1(D)$ be nonnegative such that $\op y \notin L^1(D)$. Let $\tau$ be a random time such that $E\op y_\tau =\infty$ and define $u=\delta_{\tau}$. We have $u \in L^\infty(M)$, but
\[
E\int ydu^o = E\int \op y du = E \op y_\tau = \infty,
\]
so $u^o\notin L^\infty(M)$. 
\end{example}

\section{Optional projection under Doob property}\label{sec:doob}

This section studies the case where the optional projection is a continuous linear mapping of the space $\Y(D)$ to itself. Without loss of generality, we assume that the collection $\P$ of seminorms forms a nondecreasing sequence. Continuity of the projection then means that for each $p\in\P$ there exists a $p'\in\P$ and a constant $q$ such that
\[
p(\|\op y\|)\le qp'(\|y\|)
\]
for all $y\in\Y(D)$. It turns out that this holds if and only if for each $p\in\P$ there exists a $p'\in\P$ and a constant $q$ such that
\[
p(\|m\|)\le qp'(m_T)
\]
for every martingale $m$. When this holds, we say that $\Y$ has the {\em Doob property}.

An Orlicz space has the Doob property if the conjugate of the defining Young function satisfies the $\Delta_2$-condition; see \cite[Section~VI.103]{dm82}. This generalizes the better known case of $\Y=L^p$ with $p>1$. More generally, we have the following characterization of Banach function spaces with Doob property; see \cite{kik99}.

Recall that a seminorm $p$ is {\em rearrangement invariant} if $p(\eta)=p(\eta')$ for all $\eta,\eta'\in L^0$ whose distributions coincide.

\begin{example}\label{ex:boyd}
Let $\Y$ be a rearrangement invariant Banach function space with the norm $p$ on an atomless probability space. The upper Boyd index is given by 
\[
\beta := \lim_{s\upto\infty}\frac{\log \|D_{1/s}\|_{\bar p}}{\log s},
\]
where $D:L^1([0,1])\to L^1([0,1])$ is the dilation operator defined by
\[
D_s q(t):=\begin{cases}
q(st)\quad&\text{if } st\le 1\\
0\quad&\text{if } st>1,
\end{cases}
\]
$\bar p$ is a sublinear symmetric function on $L^1([0,1])$ such that $p(\eta)=\bar p(q_{|\eta|})$ for all $\eta\in L^1$ and 
\[
\|D_{s}\|_{\bar p} := \sup_{q\in L^1([0,1])}\{ \bar p(D_{s}q) \mid \bar p(q)\le 1\} 
\]
is the operator norm of $D_{s}$.

There exists a constant $q$ such that
\[
p(\|m\|)\le q p(m_T)
\]
for every martingale $m$ if and only if $\beta<1$.

\end{example}

In this section, we define $\D$ as the optional processes in $\Y(D)$. Since convergence in $\Y(D)$ implies convergence almost surely everywhere, $\D$ is a closed subspace of $\Y(D)$. The first statement of Theorem~\ref{thm:raw} thus gives the following.

\begin{lemma}
The space $\D$ is Fr\'echet.
\end{lemma}

When the optional projection $\pi$ is continuous on $\Y(D)$, it has an {\em adjoint} $\pi^*$ which is a continuous linear operator on the dual $\U(M\times\tilde M)$ of $\Y(D)$ defined by
\[
\langle \pi y, (u,\tilde u)\rangle =\langle y, \pi^*(u,\tilde u)\rangle \quad \forall y \in \Y(D),\ \forall (u,\tilde u)\in\U(M\times\tilde M). 
\]

\begin{theorem}\label{thm:doob}
The conditions
\begin{enumerate}
\item[(a)]
$\Y$ has the Doob property,
\item[(b)]
the optional projection is continuous on $\Y(D)$,
\end{enumerate}
are equivalent and imply that the adjoint of the optional projection is given by
\[
\pi^*(u,\tilde u) = (u^o+(\tilde u^p)^c,(\tilde u^p)^d).
\]
and that the dual of $\D$ can be identified with $\hat\M$ through the bilinear form
\[
\langle y,(u,\tilde u)\rangle = E\left[\int ydu+\int y_-d\tilde u\right].
\]
The topology of $\D$ is generated by the seminorms
\[
 p_\D(y):= \inf_{z\in\Y(D)} \{p(\|z\|) \mid \op z = y\} \quad p\in \P
\]
whose polars are given by
\[
 p_\D^\circ ((u,\tilde u))= p^\circ(\|u\|+\|\tilde u\|).
\]
\end{theorem}

\begin{proof}
As already noted, (b) implies (a). To prove the converse, let $y\in\Y(D)$ and $m=\op(\one\|y\|)$. Since $|y|\le\one\|y\|$, we have $|\op y|\le\op|y|\le m$, so $\|\op y\|\le\|m\|$, while (a) gives
\[
p(\|m\|)\le q p'(\|y\|).
\]
The monotonicity of $p$ now gives (b).

If $y\in L^\infty(D)$, Lemma~\ref{lem:oppp} gives for all $(u,\tilde u)\in\U(M\times\tilde M)$,
\begin{align}
\langle \op{y},(u,\tilde u)\rangle &= E\left[\int\op ydu+\int (\op{y})_-d\tilde u\right]\nonumber\\
&= E\left[\int\op ydu+\int \pp{(y_-)}d\tilde u\right]\nonumber\\
&= E\left[\int ydu^o+\int y_-d\tilde u^p\right]\nonumber\\
&= E\left[\int yd(u^o+(\tilde u^p)^c)+\int y_-d(\tilde u^p)^d\right].\label{e:adjoint}
\end{align}
Let $p\in\P$ be such that $p^\circ(\|u\|+\|\tilde u\|)<\infty$ . Under (b), \eqref{e:adjoint} and Theorem~\ref{thm:raw} give
\begin{align*}
E\left[\int yd(u^o+(\tilde u^p)^c)+\int y_-d(\tilde u^p)^d\right] &\le p(\|\op{y}\|)p^o(\|u\|+\|\tilde u\|)\\
&\le qp'(\|y\|)p^o(\|u\|+\|\tilde u\|).
\end{align*}
Taking the supremum over $\{y\in L^\infty(D)\,|\,p'(\|y\|)\le 1\}$, gives, by Theorem~\ref{thm:raw},
\[
(p')^\circ(\|u^o+(\tilde u^p)^c\|+\|(\tilde u^p)^d\|)\le qp^\circ(\|u\|+\|\tilde u\|).
\]
Thus $\|u^o+(\tilde u^p)^c\|+\|(\tilde u^p)^d\|\in\U$, so $(u^o+(\tilde u^p)^c,\tilde u^p)\in\U(M\times\tilde M)$. Thus the density of $L^\infty(D)$ in $\Y(D)$ implies that \eqref{e:adjoint} extends to all of $\Y(D)$, so the adjoint is given by
\[
\pi^*(u,\tilde u) = (u^o+(\tilde u^p)^c,(\tilde u^p)^d).
\]
Clearly, $\pi$ is a surjection to $\D$, so, by Theorem~\ref{thm:abs}, the dual of $\D$ can be identified with $\hat \M$.
\end{proof}

The characterization of the dual of $\D$ in Theorem~\ref{thm:abs} generalizes \cite[Theorem~VII.65]{dm82} and \cite[Theorem~3.1]{ara14} that dealt with $L^p$ and Morse hearts of Orlicz spaces, respectively. Indeed, \cite[pages~166--169]{dm82} establish the Doob inequality when $\Y$ is the Orlicz space associated with a Young function whose conjugate has the $\Delta_2$-property. In that case, we may apply Theorem~\ref{thm:abs} in the setting of Example~\ref{ex:orl}. Example~\ref{ex:boyd} extends this to symmetric Banach function spaces.






\section{Optional projection under Choquet property}\label{sec:choquet}

When $\Y$ fails to have the Doob property, it may happen that $y^o\notin\Y(D)$ for an $y\in\Y(D)$. Nevertheless, if
\begin{equation}\label{eq:jen0}
p(E_\tau\xi)\le p(\xi)\quad \forall\ \xi\in\Y,\ \tau\in\T
\end{equation}
for all $p\in\P$, then
\begin{align}\label{eq:jen}
\sup_{\tau\in\T}p(\op y_\tau)\le p(\|y\|)\quad \forall p\in\P
\end{align}
for all $y\in\Y(D)$. This means that the optional projection of a $y\in\Y(D)$ belongs to the space $\tilde\D$ of optional cadlag processes for which the seminorms
\[
p_\T(y):=\sup_{\tau\in\T}p(y_\tau)
\]
are finite for all $p\in\P$. We will assume \eqref{eq:jen0} and equip $\tilde\D$ with the topology induced by the $p_\T$ and define $\D$ as the closure in $\tilde\D$ of the space $\D^\infty$ of bounded optional cadlag processes. It was shown in \cite[Remark~15]{pp20} that \eqref{eq:jen0} holds whenever the underlying probability space is resonant and $p$ is rearrangement invariant.


It was shown in \cite{bis78} and \cite[Section VI.1]{dm82} that when $\Y=L^1$, the space $\tilde\D$ is complete. The following extends this to general $\Y$.

\begin{lemma}\label{lem:dtilde}
The spaces $\tilde\D$ and $\D$ are Fr\'echet and the elements of $\D$ are of class $(D)$.
\end{lemma}

\begin{proof}

We start by showing that $\tilde\D$ is complete. If $(y^\nu)$ is a Cauchy sequence in $\tilde\D$, it is, by Property 2, Cauchy also in $\D^1$ of optional cadlag processes equipped with the norm $\sup_{\tau\in\T}E|y_\tau|$. By \cite[Theorem~VI.22]{dm82}), $\D^1$ is complete, so $(y^\nu)$ $\D^1$-converges to an $y\in \D^1$. Being Cauchy in $\tilde\D$ means that for every $\epsilon>0$ and $p\in\P$, there is an $N$ such that
\[
p_\T(y^\nu-y^\mu)\le\epsilon\quad\forall\nu,\mu\ge N.
\]
By the triangle inequality and property 3 of $p$,
\[
p_\T(|y^\nu-y|-|y-y^\mu|)\le\epsilon\quad\forall\nu,\mu\ge N.
\]
Letting $\mu\to\infty$ and using property 1 (and the fact that pointwise supremum of lsc functions is lsc) now gives
\[
p_\T(y^\nu-y)\le\epsilon\quad\forall\nu\ge N.
\]
Since $p\in\P$ and $\epsilon>0$ were arbitrary, we thus have $y\in\tilde\D$ and that $(y^\nu)$ converges in $\tilde\D$ to $y$. Thus $\tilde\D$ is complete. Since $\D$ is a closed subspace of a Fr\'echet space, it is Fr\'echet as well.


Given $y\in \D$ and $\epsilon>0$, there exists $y^\epsilon \in\D^\infty$ such that $\sup_{\tau\in\T}E|y_\tau-y^\epsilon_\tau|<\epsilon/2$. By Chebyshev's inequality,
\begin{align*}
\sup_{\tau\in\T}E[|y_\tau|\one_{\{|y_\tau|\ge\nu\}}] &\le  \sup_{\tau\in\T}E|y_\tau-y^\epsilon_\tau|+ \sup_{\tau\in\T}E[|y^\epsilon_\tau|\one_{\{|y_\tau|\ge\nu\}}]\\
& \le \epsilon/2 + \|y^\epsilon\|_{L^\infty} \sup_{\tau\in\T} E|y_\tau|/\nu< \epsilon
\end{align*}
for $\nu$ large enough, which shows that $y$ is of class $(D)$.
\end{proof}

We say that $p$ has the {\em Choquet property} if $p^\circ$ is a {\em Choquet integral} in the sense that, for every  $\eta\in\U$,
\[
p^\circ(\eta) = \int_0^\infty p^\circ(\one_{\{|\eta|\ge s\}})ds.
\]
This is clearly satisfied if $\Y=L^1$ or $\Y=L^\infty$ (although the latter fails property~4 in Section~\ref{sec:frechet}). More generally, we have the following.

\begin{example}
In the setting of Example~\ref{ex:ml}, the Marcinkiewicz norm has the Choquet property and it satisfies \eqref{eq:jen0}. If a rearrangement invariant norm has the Choquet property, then it is equivalent to a Marcinkiewicz norm. 
\end{example}
\begin{proof}
Let $p$ be the Marcinkiewicz norm. Since $p^\circ(1_A)=\phi(P(A))$, we get
\begin{align*}
p^\circ(\eta)&=\int_0^1 q_{|\eta|}(t)d\phi(t)\\
 &=-\int_0^1 \phi(t)dq_{|\eta|}(t)\\
& = \int_0^\infty \phi(n_{|\eta|}(s))ds \\
& = \int_0^\infty \phi(P(\{|\eta|\ge s\}))ds \\
& = \int_0^\infty p^\circ(\one_{\{|\eta|\ge s\}})ds, 
\end{align*}
where the first equality follows from Example~\ref{ex:ml}, the second from integration by parts and the third from \cite[Theorem~VI.55]{dm82}. Thus $p^\circ$ has the Choquet property. Since the Marcinkiewicz norm is rearrangement invariant, it satisfies \eqref{eq:jen0} by \cite[Remark~15]{pp20}. 

Assume now that $p$ is rearrangement invariant with the Choquet property. By \cite[Proposition 2.4.2]{bs88}, $p^\circ$ is rearrangement invariant, so there exists $\tilde\phi$ such that $p^\circ(1_A)= \tilde\phi(P(A))$ for all $A\in\F$. Since $p$ has the Choquet property,
\begin{align*}
p^\circ(\eta)= \int_0^\infty p^\circ(\one_{\{|\eta|\ge s\}})ds = \int_0^\infty \tilde\phi(n_{|\eta|}(s))ds &=-\int_0^1 \tilde\phi(t)dq_{|\eta|}(t)= \int_0^1 q_{|\eta|}(t)d\tilde\phi(t),
\end{align*}
where the third equality follows from \cite[Theorem~VI.55]{dm82} and the last from integration by parts. By \cite[Theorem 4.7]{kps82}, $\tilde\phi$ is quasiconcave in the sense of \cite[Definition 1.1]{kps82}, so, by \cite[Corollary~1.1]{kps82}, there exists a concave $\phi$ and $C\in\reals_+$ such that 
\[
\frac{1}{C}\phi(t)\le\tilde\phi(t)\le C\phi(t).
\]
Thus $p^\circ$ is equivalent to the Lorentz norm $\int_0^1q_{|\eta|}(t)d\phi(t)$, and hence $p$ is equivalent to a Marcinkiewicz norm.
\end{proof}


The following is an extension of \cite{sch86}; see also \cite[Section~4.5]{fs16}.

\begin{lemma}\label{lem:ma}
A real-valued function $\rho$ on $\U$ with $\rho(1)=1$ is a Choquet integral on $\U$ if and only if it is monotone, 
\[
\rho(\eta\wedge\nu)\nearrow \rho(\eta)\quad\forall\eta\in\U_+,
\]
and comonotone additive in the sense that
\[
\rho(\eta+\eta')=\rho(\eta)+\rho(\eta')
\] 
whenever $\eta,\eta'\in\U$ satisfy $(\eta(\omega)-\eta(\omega'))(\eta'(\omega)-\eta'(\omega'))\ge 0$ for all $\omega,\omega'\in\Omega$.
\end{lemma}

\begin{proof}
The necessity is proved as in the proof of \cite[Theorem~4.88]{fs16} (their argument does not require $\U=L^\infty$). As to sufficiency, \cite[Theorem]{sch86} says that, $\rho$ is Choquet integral on $L^\infty$. By monotone convergence,
\[
\int_0^\infty \rho(\one_{\{\eta\wedge \nu\ge s\}})ds \nearrow \int_0^\infty \rho(\one_{\{\eta\ge s\}})ds
\]
while $\rho(\eta\wedge\nu)\nearrow \rho(\eta)$, by assumption.
\end{proof}

\begin{theorem}\label{thm:choquet}
Assume that each $p\in\P$  satisfies \eqref{eq:jen0} and has the Choquet property. Then the dual of $\D$ can be identified with $\hat\M$ under the bilinear form
\[
\langle y,(u,\tilde u)\rangle := E\left[\int ydu + \int y_-d\tilde u\right].
\]
The optional projection is a continuous surjection of $\Y(D)$ to $\D$, its adjoint is the embedding of $\hat\M$ to $\U(M\times\tilde M)$,  and the topology of $\D$ is generated by the seminorms
\[
 p_\D(y):= \inf_{z\in\Y(D)} \{p(\|z\|) \mid \op z = y\} \quad p\in \P
\]
whose polars are given by
\[
 p_\D^\circ ((u,\tilde u))= p^\circ(\|u\|+\|\tilde u\|).
\]
\end{theorem}

\begin{proof}
By \eqref{eq:jen}, the optional projection is continuous from $\Y(D)$ to $\tilde\D$ with norm one. Since $L^\infty(D)$ is dense in $\Y(D)$, the continuity of the projection implies that its range is contained in $\D$. By Theorem~\ref{thm:abs}, it suffices to show that $\hat\M$ is the dual of $\D$.

Let $y\in\D$ and $(u,\tilde u)\in\hat\M$ and denote the corresponding total variations processes by $u^{TV}$ and $\tilde u^{TV}$. By \cite[Theorem~IV.50]{dm78}, $\tau_s=\inf\{t\,|\, u^{TV}_t\ge s\}$ is a stopping time, and, by \cite[A on page~xiii]{dm82}, $\tilde\tau_s=\inf\{t\,|\, \tilde u^{TV}_t\ge s\}$ is a predictable time. By \cite[Theorem~55]{dm82} and Fubini-Tonelli,
\begin{align*}
E[\int ydu+\int y_-d\tilde u] &\le E\int (|y|du^{TV}+|y_-|d\tilde u^{TV})\\
&=E\int_0^{\infty}(|y_{\tau_s}|\one_{\{\|u\|\ge s\}}+|y_{\tilde \tau_s-}|\one_{\{\|\tilde u\|\ge s\}})ds\\
&=\int_0^\infty(E\left[|y_{\tau_s}|\one_{\{\|u\|\ge s\}}+|y_{\tilde \tau_s-}|\one_{\{\|\tilde u\|\ge s\}})\right]ds\\
&\le \int_0^\infty[p(y_{\tau_s})p^\circ(\one_{\{\|u\|\ge s\}})+p(y_{\tilde\tau_s-})p^\circ(\one_{\{\|\tilde u\|\ge s\}})]ds.
\end{align*}
By \cite[Theorem~4.16]{hwy92}, there is a sequence $(\tau^\nu)$ of stopping times with $\tau^\nu<\tau^{\nu+1}$ and $\tau^\nu\nearrow\tilde\tau_s$ almost surely. Since $y$ is of class $(D)$ and $p$ is weakly lsc in $L^1$, we get $p(y_{\tilde \tau_s-})\le\liminf_\nu p(y_{\tau^\nu})\le\sup_{\tau\in\T}p(y_\tau)$. By Choquet property,
\begin{align*}
E[\int ydu+\int y_-d\tilde u] &\le\sup_{\tau\in\T}p(y_\tau)\int_0^\infty[p^\circ(\one_{\{\|u\|\ge s\}})+p^\circ(\one_{\{\|\tilde u\|\ge s\}})] ds\\
&=p_\T(y)[p^\circ(\|u\|)+p^\circ(\|\tilde u\|)]\\
&\le 2p_\T(y)p^\circ(\|u\|+\|\tilde u\|).
\end{align*}
Thus $(u,\tilde u)$ defines a continuous linear functional on $\D$. 


On the other hand, let $J$ be a continuous linear functional on $\D$. The continuity implies that $J$ is continuous on $\D^\infty\subseteq\D\cap\Y(D)$ also with respect to the relative topology of $\Y(D)$. By Hahn--Banach, $J$ extends to a continuous linear functional on all of $\Y(D)$. Theorem~\ref{thm:raw} then gives the existence of a $(w,\tilde w)\in\U(M\times\tilde M)$ such that
\[
J(y) = E\left[\int ydw+\int y_-d\tilde w\right]\quad\forall y\in\D^\infty.
\]
By the definitions of the projections,
\[
J(y) = E\left[\int ydw^o+\int y_-d\tilde w^p\right] = E\left[\int ydu+\int y_-d\tilde u\right]\quad\forall y\in\D^\infty,
\]
where $(u,\tilde u):=(w^o+(\tilde w^p)^c,(\tilde w^p)^d)\in L^1(M\times\tilde M)$ with $u$ optional and $\tilde u$ predictable. The continuity of $J$ on $\D$ means that there is a $p\in\P$ such that
\[
p_\T^\circ(J):=\sup_{y\in\D}\left\{J(y) \midb p_\T(y)\le 1\right\}<\infty.
\]
By \eqref{eq:pptv}, Lemma~\ref{lem:oppp} and \eqref{eq:jen}, 
\begin{align*}
p^\circ(\|u\|+\|\tilde u\|) &= \sup_{y\in L^\infty(D)}\left\{E\left[\int ydu +\int y_-d\tilde u \right] \midb p(\|y\|)\le 1\right\}\\
&=\sup_{y\in L^\infty(D)}\left\{E\left[\int\op ydu +\int\pp(y_-)d\tilde u \right] \midb p(\|y\|)\le 1\right\}\\
&\le\sup_{y\in L^\infty(D)}\left\{E\left[\int\op ydu +\int(\op y)_-d\tilde u \right] \midb p_\T(\op y)\le 1\right\}\\
&=\sup_{y\in L^\infty(D)}\left\{J(\op y) \midb p_\T(\op y)\le 1\right\}\le p_\T^\circ(J),
\end{align*}
where the last equality holds since $\op y\in\D^\infty$ for all $y\in L^\infty(D)$. Thus, $J$ is represented on $\D^\infty$ by a $(u,\tilde u)\in\hat\M$. By continuity, the representation is valid on all of $\D=\cl\D^\infty$.
\end{proof}

It is clear from the above proof that, instead of the Choquet property, it would suffice that $\int_0^\infty p^\circ(\one_{\{\eta\ge s\}})ds$ is finite whenever $\eta\in\dom p^\circ$.

When $\Y=L^1$, Theorem~\ref{thm:choquet} can be written as follows.

\begin{corollary}\label{cor:L1}
The space $\D^1$ of optional cadlag processes of class $(D)$ equipped with the norm
\[
\|y\|_{\D^1}:=\sup_{\tau\in\T}E|y_\tau|
\]
is Banach and its dual can be identified with $\hat \M^\infty$ through the bilinear form
\[
\langle y,(u,\tilde u)\rangle =E\left[\int ydu+\int y_-d\tilde u\right].
\]
The optional projection is a continuous surjection of $L^1(D)$ to $\D^1$ and its adjoint is the embedding of $\hat\M^\infty$ to $L^\infty(M\times \tilde M)$. The topology of $\D^1$ is generated by the seminorm
\[
 p_\D(y):= \inf_{z\in L^1(D)} \{E\|z\| \mid \op z = y\}
\]
whose polar is given by
\[
 p_\D^\circ ((u,\tilde u))= \esssup(\|u\|+\|\tilde u\|).
\]
\end{corollary}

\begin{proof}
Since $\Y=L^1$ has the Choquet property, it suffices, by Theorem~\ref{thm:choquet}, to check that $\D^1$ is the closure of $\D^\infty$ in $\tilde\D^1$. Let $y\in\D^1$ and define $y^\nu\in\D^\infty$ as the pointwise projection of $y$ to the Euclidean unit ball of radius $\nu=1,2,\ldots$. By uniform integrability,
\[
\sup_{\tau\in\T} E|y_\tau-y^\nu_\tau|\le\sup_{\tau\in\T}E[|y_\tau|\one_{\{|y_\tau|\ge\nu\}}]\rightarrow 0,
\]
so $y\in\cl\D^\infty$.
\end{proof}

Corollary~\ref{cor:L1} complements \cite[Theorem~67]{dm82} which characterizes the dual of the Banach space of cadlag processes whose pathwise sup-norm is integrable. The larger space $\D^1$ in Corollary~\ref{cor:L1} was studied in \cite[Section~VI.1]{dm82}. The above characterization of its dual seems new. The surjectivity of the projection in Corollary~\ref{cor:L1} was stated in \cite[Theorem~4]{bis78} without a complete proof.



\section{Regular processes}\label{sec:regular}\label{sec:regular}

Following \cite{bis78} we say that an adapted cadlag process $y$ of class $(D)$ is {\em regular} if 
\[
\pp y=y_-.
\]
According to Bismut~\cite[Theorem~3]{bis78}, regular processes are the optional projections of elements of $L^1(C)$. This section gives an easy derivation of Bismut's result while allowing for more general $\Y$ in place of $L^1$. We assume that $\Y$ is as in Section~\ref{sec:frechet} and define
\begin{align*}
\Y(C) &:=\{y\in L^1(C)\mid \|y\|\in\Y\}.
\end{align*}

The following specializes Theorem~\ref{thm:raw} to continuous processes.

\begin{corollary}\label{cor:raw}
The space $\Y(C)$ is Fr\'echet and its dual can be identified with
\[
\U(M) := \{u\in L^1(M)\mid \|u\|\in\U \}
\]
through the bilinear form
\[
\langle y,u\rangle := E\int ydu.
\]
For every $y\in L^1(C)$ and $u\in L^1(M)$,
\begin{equation*}
E\int ydu\le p(\|y\|)p^\circ(\|u\|)
\end{equation*}
and
\begin{equation*}
p^\circ(\|u\|) = \sup_{y\in L^\infty(C)}\left\{E\int ydu \midb p(\|y\|)\le 1\right\}.
\end{equation*}
In particular, $u\mapsto p^\circ(\|u\|)$ is the polar of $y\mapsto p(\|y\|)$.
\end{corollary}


\begin{proof}
$\Y(C)$ is a closed subspace of $\Y(D)$ and thus Fr\'echet. The elements of $\U(M)$ define continuous linear functionals on $\Y(C)$. On the other hand, by Hahn-Banach, a continuous linear functional $l$ on $\Y(C)$ extends to a continuous linear functional on $\Y(D)$, which, by Theorem~\ref{thm:raw}, has the expression
\[
l(y)=E[\int ydu+\int y_-d\tilde u]
\]
for some $(u,\tilde u)\in\U(M\times\tilde M)$. On $\Y(C)$, this can be written as 
\[
l(y)=E\int yd(u+\tilde u),
\]
where $u+\tilde u\in \U(M)$. The expression for the polar seminorm follows as in the proof of Theorem~\ref{thm:raw}.
\end{proof}

We will assume that one of the equivalent conditions in Theorem~\ref{thm:abs} is satisfied and denote
\[
\R := \{y\in\D\,|\,\pp y=y_-\}.
\]
We endow $\R$ with the relative topology it has as a subspace of $\D$. Let
\[
\M :=\{u\in\U(M)\mid u\text{ optional}\}.
\]
The following is proved like Lemma~\ref{lem:opm} except that instead of Theorem~\ref{thm:raw} one applies Corollary~\ref{cor:raw}.

\begin{lemma}\label{lem:opm2}
The space $\M$ is the orthogonal complement of $\ker\pi$ and thus, weakly closed in $\U(M)$.
\end{lemma}

Combining this with Theorem~\ref{thm:abs} and the Hahn--Banach theorem, gives the following.

\begin{theorem}\label{thm:reg}
Under the assumptions of Theorem~\ref{thm:abs}, $\R$ is Fr\'echet and its dual can be identified with $\M$ under the bilinear form
\[
\langle y,u\rangle := E\int ydu.
\]
The optional projection is a continuous surjection of $\Y(C)$ to $\R$, its adjoint is the embedding of $\M$ to $\U(M)$ and $(\ker\pi)^\perp=\M$. Moreover, the topology of $\R$ is generated by the seminorms
\[
p_\R(y) = \inf_{z\in\Y(C)}\{p(\|z\|)\,|\,\op z=y\}
\]
the polars of which are given by
\[
p_\R^\circ(u) = p^\circ(\|u\|).
\]
\end{theorem}

\begin{proof}
We start by showing that $\R$ is the orthogonal complement (with respect to the pairing of $\D$ and $\hat\M$) of the linear space
\[
\L = \{(u,\tilde u)\in\hat\M\,|\, u+\tilde u=0\}.
\]
If $y\in\R$ and $(u,\tilde u)\in\hat\M$, we have
\[
E\left[\int ydu+\int y_-d\tilde u\right] =E\left[\int ydu+\int \pp y d\tilde u\right] = E\int yd(u+\tilde u),
\]
so $\R\subseteq\L^\perp$. On the other hand, if $y\in\D\setminus\R$, there exists, by the predictable section theorem, a predictable time $\tau$ such that $E(\pp y_\tau-y_{\tau-})\ne 0$. Defining $u=-\tilde u=\delta_\tau$, we have $(u,\tilde u)\in\L$ while $\langle y,(u,\tilde u)\rangle=E(\pp y_\tau-y_{\tau-})$. Thus, $\R=\L^\perp$.

Being a closed subspace of a Fr\'echet space, $\R$ is Fr\'echet. 
Since $\M$ is isomorphic to a subspace of $\hat\M$, every $u\in\M$ defines a continuous linear functional on $\R$. On the other hand, by Hahn-Banach, a continuous linear functional on $\R$  extends to a continuous linear functional $l$ on $\D$ which, by assumption, has the expression
\[
l(y)=E\left[\int ydu+\int y_-d\tilde u\right]
\]
for some $(u,\tilde u)\in\hat\M$. On $\R$, this can be expressed as
\[
E\left[\int ydu+\int y_-d\tilde u\right] = E\left[\int ydu+\int \pp y d\tilde u\right] = E\int yd(u+\tilde u),
\]
so the dual of $\R$ can indeed be identified with $\M$. 

If $y\in\Y(C)$, we have $\op y\in\D$, by assumption, and then, by Lemma~\ref{lem:oppp},
\[
(\op y)_-=\pp (y_-)=\pp y= \pp(\op y),
\]
so $\op y\in\R$. By Lemma~\ref{lem:opm2}, the density of $L^\infty(C)$ in $\Y(C)$ and the continuity of $\pi$ imply $(\ker\pi)^\perp=\M$.

The claims about the surjectivity of $\pi$, its adjoint and the seminorms are established like in the proof of Theorem~\ref{thm:abs}.
\end{proof}


When $\Y=L^1$, Corollary~\ref{cor:L1} implies that the assumptions of Theorem~\ref{thm:abs} hold, so Theorem~\ref{thm:reg} gives the following refinement of Corollary~\ref{cor:L1}, first derived in \cite{pp17} using the main result of \cite{bis78}.

\begin{corollary}\label{cor:reg1}
The space $\R^1$ of regular processes equipped with the norm
\[
\|y\|_{\D^1}:=\sup_{\tau\in\T}E|y_\tau|
\]
is Banach and its dual can be identified with $\M^\infty$ through the bilinear form
\[
\langle y,u\rangle =E\int ydu.
\]
The optional projection is a continuous surjection of $L^1(C)$ to $\R^1$ and its adjoint is the embedding of $\M^\infty$ to $L^\infty(M)$. The topology of $\R^1$ is generated by the seminorm
\[
 p_\D(y):= \inf_{z\in L^1(C)} \{E\|z\| \mid \op z = y\}
\]
whose polar is given by
\[
 p_\D^\circ (u)= \esssup(\|u\|).
\]
\end{corollary}

Corollary~\ref{cor:reg1} applies Theorem~\ref{thm:reg} to $\Y$ with the Choquet property. Likewise, Theorem~\ref{thm:reg} could be applied to cases when $\Y$ has the Doob property. This would cover appropriate Orlicz spaces and the Fr\'echet space of random variables with finite moments.

\section{Doob decomposition}\label{sec:dd}

We will say that an optional cadlag process $Z$ is a {\em $\U$-quasimartingale} if for some $p\in\P$,
\[
\Var_p(Z):=\sup_{(\tau_i)_{i=0}^n\subset\T}p^\circ\left(\sum_{i=0}^{n-1}|E_{\tau_i}[Z_{\tau_i}-Z_{\tau_{i+1}}]| + |Z_{\tau_n}|\right)<\infty.
\]
When $\U=L^1$ and $p^\circ(\xi)=\|\xi\|_{L^1}$, this reduces to the usual definition of a quasimartingale; see e.g.\ \cite[Definition~VI.38]{dm82}. The space of quasimartingales contains e.g.\ supermartingales and their differences. The classical Doob-decomposition expresses a quasimartingale of class $(D)$ as a sum of a martingale and a predictable process of integrable variation; see e.g.~\cite[Appendix~II.4]{dm82}. Choosing $n=0$, we see that a $\U$-quasimartingale is of class $(D)$ as soon as the level sets of some $p^\circ$ are uniformly integrable. 

We will denote by $\tilde\N^\U_0$ the linear space of predictable cadlag processes that start at $0$ and whose pathwise variation is in $\U$. The theorem below gives a refined Doob decomposition for $\U$-quasimartingales. It assumes that the seminorms satisfy the ``Jensen inequality'' \eqref{eq:jen0}.

\begin{lemma}\label{lem:jen}
If $p$ satisfies \eqref{eq:jen0}, then $p^\circ$ satisfies it as well.
\end{lemma}

\begin{proof}
By properties of the conditional expectation,
\begin{align*}
  p^\circ(E_\tau\eta) &=\sup_{\xi\in L^\infty}\{E[\xi E_\tau\eta]\mid p(\xi)\le 1\} =\sup_{\xi\in L^\infty}\{E[\eta E_\tau\xi]\mid p(\xi)\le 1\}\\
  &\le \sup_{\xi\in L^\infty}\{E[\eta E_\tau\xi]\mid p(E_\tau\xi)\le 1\} \le \sup_{\xi\in L^\infty}\{E[\eta \xi]\mid p(\xi)\le 1\} = p^\circ(\eta),
\end{align*}
for any $\tau\in\T$.
\end{proof}

\begin{theorem}\label{thm:dt}
Assume that the conditions of Theorem~\ref{thm:abs} are satisfied, that the seminorms $p\in\P$ satisfy \eqref{eq:jen0} and that the level sets of some $p^\circ$ are uniformly integrable. Then a process $Z$ is a $\U$-quasimartingale if and only if there exists a $\U$-martingale $M$ and an $A\in\tilde\N^\U_0$ such that 
\[
Z_t = M_t-A_t.
\]
\end{theorem}

\begin{proof}
If $Z=M-A$ for a $\U$-martingale $M$ and $A\in\tilde\N^\U_0$, then the monotonicity of $p^\circ$, the Jensen's inequality with $|\cdot|$ and Lemma~\ref{lem:jen} give
\begin{align*}
  \Var_p(Z) &=\sup_{(\tau_i)_{i=0}^n\subset\T}p^\circ\left(\sum_{i=0}^{n-1}|E_{\tau_i}[A_{\tau_i}-A_{\tau_{i+1}}]| + |M_{\tau_n}-A_{\tau_n}|\right)\\
  &\le\sup_{(\tau_i)_{i=0}^n\subset\T}p^\circ\left(\sum_{i=0}^{n-1}E_{\tau_i}|A_{\tau_i}-A_{\tau_{i+1}}| + |M_{\tau_n}|+|A_{\tau_n}|\right)\\
&\le p^\circ(2\|A\|_{TV}+|M_{\tau_n}|) \\
&\le p^\circ(2\|A\|_{TV})+ p^\circ(M_\infty)<\infty,
\end{align*}
where $\|A\|_{TV}$ denotes the total variation of $A$. On the other hand, let $\D_s\subset\D$ be the space of {\em simple processes} of the form
\[
y = \sum_{i=0}^n\one_{[\tau_i,\tau_{i+1})}\eta^i,
\]
where $(\tau_i)_{i=0}^n$ is an increasing sequence of stopping times with $\tau_0=0$ and $\tau_{n+1}=\infty$ and $\eta^i\in L^\infty(\F_{\tau_i})$. Define a linear functional $l$ on $\D_s$ by
\[
l(y) = E\left[\sum_{i=0}^{n-1}y_{\tau_i}E_{\tau_i}[Z_{\tau_i}-Z_{\tau_{i+1}}] + y_{\tau_n}E_{\tau_n}Z_{\tau_n}\right].
\]
Given $\bar y\in\D_s$, 
\begin{align*}
  l(\bar y) &= E\int \bar yd\bar u^{(\tau_i)},
\end{align*}
where the measure $\bar u^{(\tau_i)}$ is given by
\[
\bar u^{(\tau_i)}:=\sum_{i=0}^nE_{\tau_i}[Z_{\tau_i}-Z_{\tau_{i+1}}]\delta_{\tau_i} + E_{\tau_n}Z_{\tau_n}\delta_{\tau_n}.
\]
Thus, by Theorem~\ref{thm:abs},
\[
l(\bar y) \le p_{\D}(\bar y)p^\circ(\|\bar u^{(\tau_i)}\|) \le p_{\D}(\bar y)\Var_p(Z)
\]
so, $l$ is continuous in the relative topology of $\D_s$. By Hahn--Banach, $l$ extends to all of $\D$ so by, Theorem~\ref{thm:abs}, there exists $(u,\tilde u)\in\hat\M$ such that
\[
l(y)=E\left[\int ydu + \int y_-d\tilde u\right].
\]

Given $\tau\in\T$, let $y^\nu=\one_{[\tau,\tau+1/\nu)}\in\D_s$. Since $Z$ is cadlag and of class $(D)$, we have $l(y^\nu)=E(Z_{\tau}-Z_{\tau+1/\nu})\to 0$. On the other hand,
\[
l(y^\nu)=E\left[\int y^\nu du + \int y^\nu_-d\tilde u\right]=E[u([\tau,\tau+1/\nu)) + \tilde u((\tau,\tau+1/\nu])] \to Eu(\{\tau\})
\]
so $Eu(\{\tau\})=0$ for every $\tau\in\T$. Thus the purely discontinuous part of $u$ is zero and, in particular, $u$ is predictable. We can thus express $l$ in terms of the predictable measure $\bar u:=u+\tilde u$ as
\[
l(y)=\int y_-d\bar u.
\]
It now suffices to take $A_t=u((0,t])$ and $M_t=E[A_\infty\,|\,\F_t]$. Indeed, taking $y=\one_{[\tau,\infty)}\in\D_s$, gives
\[
E(Z_\tau\one_{\{\tau<\infty\}}) = E(A_\infty-A_\tau).
\]
Taking $\tau=\tau_B$ for $B\in\F_\tau$ gives $Z_\tau=E[A_\infty-A_\tau\,|\,\F_\tau]$. Here $\tau_B :=\tau$ on $B$ and $+\infty$ otherwise. Finally, Lemma~\ref{lem:jen} gives $p^\circ(M_t)=p^\circ(E_tA_\infty)\le p^\circ(A_\infty)$, so $M$ is a $\U$-martingale.
\end{proof}

\bibliographystyle{plain}
\bibliography{sp}

\begin{thebibliography}{10}

\bibitem{ara14}
T.~Arai.
\newblock Convex risk measures for c\`adl\`ag processes on {O}rlicz hearts.
\newblock {\em SIAM J. Financial Math.}, 5(1):609--625, 2014.

\bibitem{bs88}
C.~Bennett and R.~Sharpley.
\newblock {\em Interpolation of operators}, volume 129 of {\em Pure and Applied
  Mathematics}.
\newblock Academic Press, Inc., Boston, MA, 1988.

\bibitem{bis73b}
J.-M. Bismut.
\newblock Conjugate convex functions in optimal stochastic control.
\newblock {\em J. Math. Anal. Appl.}, 44:384--404, 1973.

\bibitem{bis78}
J.-M. Bismut.
\newblock R\'egularit\'e et continuit\'e des processus.
\newblock {\em Z. Wahrsch. Verw. Gebiete}, 44(3):261--268, 1978.

\bibitem{bs77}
J.-M. Bismut and B.~Skalli.
\newblock Temps d'arr\^et optimal, th\'eorie g\'en\'erale des processus et
  processus de {M}arkov.
\newblock {\em Z. Wahrscheinlichkeitstheorie und Verw. Gebiete},
  39(4):301--313, 1977.

\bibitem{cv77}
C.~Castaing and M.~Valadier.
\newblock {\em Convex analysis and measurable multifunctions}.
\newblock Springer-Verlag, Berlin, 1977.
\newblock Lecture Notes in Mathematics, Vol. 580.

\bibitem{dm78}
C.~Dellacherie and P.-A. Meyer.
\newblock {\em Probabilities and potential}, volume~29 of {\em North-Holland
  Mathematics Studies}.
\newblock North-Holland Publishing Co., Amsterdam, 1978.

\bibitem{dm82}
C.~Dellacherie and P.-A. Meyer.
\newblock {\em Probabilities and potential. {B}}, volume~72 of {\em
  North-Holland Mathematics Studies}.
\newblock North-Holland Publishing Co., Amsterdam, 1982.
\newblock Theory of martingales, Translated from the French by J. P. Wilson.

\bibitem{die76}
J.~Dieudonn\'{e}.
\newblock {\em Treatise on analysis. {V}ol. {II}}.
\newblock Academic Press [Harcourt Brace Jovanovich, Publishers], New
  York-London, 1976.
\newblock Enlarged and corrected printing, Translated by I. G. Macdonald, With
  a loose erratum, Pure and Applied Mathematics, 10-II.

\bibitem{fs16}
H.~F\"{o}llmer and A.~Schied.
\newblock {\em Stochastic finance}.
\newblock De Gruyter Graduate. De Gruyter, Berlin, 2016.
\newblock An introduction in discrete time, Fourth revised and extended edition
  of [ MR1925197].

\bibitem{hwy92}
S.~W. He, J.~G. Wang, and J.~A. Yan.
\newblock {\em Semimartingale theory and stochastic calculus}.
\newblock Kexue Chubanshe (Science Press), Beijing; CRC Press, Boca Raton, FL,
  1992.

\bibitem{kn76}
J.~L. Kelley and I.~Namioka.
\newblock {\em Linear topological spaces}.
\newblock Springer-Verlag, New York, 1976.
\newblock With the collaboration of W. F. Donoghue, Jr., Kenneth R. Lucas, B.
  J. Pettis, E. T. Poulsen, G. B. Price, W. Robertson, W. R. Scott, and K. T.
  Smith, Second corrected printing, Graduate Texts in Mathematics, No. 36.

\bibitem{kik99}
M.~Kikuchi.
\newblock Averaging operators and martingale inequalities in rearrangement
  invariant function spaces.
\newblock {\em Canad. Math. Bull.}, 42(3):321--334, 1999.

\bibitem{kps82}
S.~G. Kre\u{\i}n, Yu.~\={I}. Petun\={\i}n, and E.~M. Sem\"{e}nov.
\newblock {\em Interpolation of linear operators}, volume~54 of {\em
  Translations of Mathematical Monographs}.
\newblock American Mathematical Society, Providence, R.I., 1982.
\newblock Translated from the Russian by J. Sz\H{u}cs.

\bibitem{lux55}
W.~A.~J. Luxemburg.
\newblock {\em Banach function spaces}.
\newblock Thesis, Technische Hogeschool te Delft, 1955.

\bibitem{lz71}
W.~A.~J. Luxemburg and A.~C. Zaanen.
\newblock {\em Riesz spaces. {V}ol. {I}}.
\newblock North-Holland Publishing Co., Amsterdam-London; American Elsevier
  Publishing Co., New York, 1971.
\newblock North-Holland Mathematical Library.

\bibitem{nov91}
I.~Novikov.
\newblock Martingale inequalities in rearrangement invariant function spaces.
\newblock In {\em Function spaces ({P}ozna\'{n}, 1989)}, volume 120 of {\em
  Teubner-Texte Math.}, pages 120--127. Teubner, Stuttgart, 1991.

\bibitem{pp17}
T.~Pennanen and A.-P. Perkki{\"o}.
\newblock Convex integral functionals of regular processes.
\newblock {\em Stochastic Processes and their Applications}, 2017.

\bibitem{pp18}
T.~Pennanen and A.-P. Perkki{\"o}.
\newblock Convex integral functionals of processes of bounded variation.
\newblock {\em Journal of Convex Analysis}, 25(1):161--179, 2018.

\bibitem{pp20}
T.~Pennanen and A.-P. Perkkiö.
\newblock Topological duals of locally convex function spaces.
\newblock {\em arXiv}, 2008.08934, 2020.

\bibitem{pes95}
W.~R. Pestman.
\newblock Measurability of linear operators in the {S}korokhod topology.
\newblock {\em Bull. Belg. Math. Soc. Simon Stevin}, 2(4):381--388, 1995.

\bibitem{sch86}
D.~Schmeidler.
\newblock Integral representation without additivity.
\newblock {\em Proc. Amer. Math. Soc.}, 97(2):255--261, 1986.

\end{thebibliography}

\end{document}